\documentclass[11pt]{article}
\usepackage{amsmath,amsthm,amssymb,color}
\usepackage[hyperfootnotes=false]{hyperref}

\def\titlerunning#1{\gdef\titrun{#1}}
\makeatletter
\def\author#1{\gdef\autrun{\def\and{\unskip, }#1}\gdef\@author{#1}}
\def\address#1{{\def\and{\\\hspace*{18pt}}\renewcommand{\thefootnote}{}%
\footnote {#1}}%
\markboth{\autrun}{\titrun}}
\makeatother
\def\email#1{\hspace*{4pt}{\em e-mail}: #1}
\def\MSC#1{{\renewcommand{\thefootnote}{}%
\footnote{\emph{Mathematics Subject Classification (2010):} #1}}}
\def\keywords#1{\par\medskip
\noindent\textbf{Keywords:} #1}

\newtheorem{theorem}{Theorem}[section]

\newtheorem{cor}[theorem]{Corollary}
\newtheorem{lemma}[theorem]{Lemma}
\newtheorem{prob}[theorem]{Open Problem}

\newtheorem{defin}[theorem]{Definition}

\theoremstyle{definition}

\numberwithin{equation}{section}

\DeclareMathOperator{\rk}{rk}
\newcommand{\gaussmnum}[3]{\left[\begin{smallmatrix}{#1}\\{#2}\end{smallmatrix}\right]_{#3}}
\def\wham{w_{\text{h}}}
\def\dham{d_{\text{h}}}
\def\drank{d_{\text{r}}}

\def\cA{\mathcal A}
\def\cC{\mathcal C}

\def\cD{\mathcal D}

\def\cF{\mathcal F}

\def\PG{{\rm PG}}

\def\F{\mathbb F}

\begin{document}

\baselineskip=16pt

\titlerunning{}

\title{Combining subspace codes}

\author{Antonio Cossidente
\and
Sascha Kurz
\and
Giuseppe Marino
\and
Francesco Pavese}

\date{}

\maketitle

\address{A. Cossidente: Dipartimento di Matematica, Informatica ed Economia, Universit{\`a} degli Studi della Basilicata, Contrada Macchia Romana, 85100, Potenza, Italy; \email{antonio.cossidente@unibas.it}
\and
S. Kurz: Mathematisches Institut, Fakult\"at f\"ur Mathematik, Physik und Informatik, Universit\"at Bayreuth, 95440, Bayreuth, Germany; \email{sascha.kurz@uni-bayreuth.de}
\and
G. Marino: Dipartimento di Matematica e Applicazioni ``Renato Caccioppoli'', Universit{\`a} degli Studi di Napoli ``Federico II'', Complesso Universitario di Monte Sant'Angelo, Cupa Nuova Cintia 21, 80126, Napoli, Italy; \email{giuseppe.marino@unina.it}
\and
F. Pavese: Dipartimento di Meccanica, Matematica e Management, Politecnico di Bari, Via Orabona 4, 70125 Bari, Italy; \email{francesco.pavese@poliba.it}
}

\MSC{Primary 51E20; Secondary 05B25, 94B65.}


\begin{abstract}
In the context of constant--dimension subspace codes, an important problem is to determine the largest possible size $A_q(n, d; k)$ of codes whose codewords are 
$k$-subspaces of $\F_q^n$ with minimum subspace distance $d$. Here in order to obtain improved constructions, we investigate several approaches to combine subspace 
codes. This allow us to present improvements on the lower bounds for constant--dimension subspace codes for many parameters, including $A_q(10, 4; 5)$, $A_q(12, 4; 4)$, 
$A_q(12, 6, 6)$ and $A_q(16, 4; 4)$. 

\keywords{constant--dimension subspace code; finite projective geometry; network coding.}
\end{abstract}

\section{Introduction}
\label{sec_introduction}
\addtocounter{footnote}{-1}
\addtocounter{footnote}{-1}

Let $V$ be an $n$-dimensional vector space over the finite field $\F_q$, $q$ any prime power. The set $S(V)$ of all subspaces of $V$, or subspaces of the projective space 
$\PG(V)=\PG(n-1,q)$, forms a metric space with respect to the {\em subspace distance} defined by $d(U,U') = \dim (U+U') - \dim (U\cap U')$. In the context of subspace codes, 
an important problem is to determine the largest possible size $A_q(n,d)$ of codes in the space $(S(V), d)$ with a given minimum distance, and to classify the corresponding 
optimal codes. The interest in these codes is a consequence of the fact that codes in the projective space have been proposed for error control in random linear network coding, 
see \cite{KK}. In this application the codewords are mostly assumed to be contained in a Grassmannian over a finite field, i.e., they all have the same vector space dimension 
$k$. These codes are referred to as \emph{constant--dimension} codes (CDCs for short) and their maximum cardinality is denoted by $A_q(n,d;k)$. 
 
Here we will consider several approaches to combine subspace codes in order to improve lower bounds for $A_q(n, d; k)$. The currently best known lower and upper bounds for 
$A_q(n,d;k)$ can be found at the online tables \url{http://subspacecodes.uni-bayreuth.de} and the associated survey \cite{tables}. For the parameters $2\le q\le 9$, 
$4\le n\le 9$, $2\le k\le \tfrac{n}{2}$, $2\le\tfrac{d}{2}\le k$ covered there, we obtain more than 200~improved constructions.

The remaining part of this paper is structured as follows. In Section~\ref{sec_preliminaries} we introduce the necessary preliminaries and in 
Section~\ref{sec_known_construction_and_bounds} we briefly review the known constructions and bounds for $A_q(n, d; k)$. Our main results, i.e., improved constructions for 
CDCs are presented in Section~\ref{sec_constructions_rank_metric} and Section~\ref{sec_combine}. 

More precisely, in Section~\ref{sec_constructions_rank_metric} we consider constructions of CDCs based on rank metric codes. The results therein provide not only a 
generalization of several recent results \cite{chen2019new,he2019construction,HC,DH,LCF,XC}, but they also offer a more general point of view with respect to techniques 
that have been previously investigated in the literature, as for instance the so called linkage construction \cite{MR3543532,ST}. In particular, by using rank metric codes 
in different variants, we are able to obtain CDCs that either give improved lower bounds for many parameters, including $A_2(12, 4; 4)$, $A_q(12, 6; 6)$, $A_q(4k, 2k; 2k)$, 
$k \ge 4$ even, $A_q(10, 4; 5)$, or whose size matches the best known lower bounds. Note that these bounds have been previously established with different approaches as the 
so-called Echelon-Ferrers construction \cite{etzion2009error,ST}.

In Section~\ref{sec_combine} we investigate a construction method introduced for specific parameters by the first, third and fourth author in \cite{cossidente2019subspace} 
and further developed by the second author in \cite{K1}. Here this approach is generalized and applied to a wide range of parameters. In particular, it enables us to obtain 
improved lower bounds for many parameters including $A_q(12, 4; 4)$, $q \ge 3$, $A_q(16, 4; 4)$. Finally we discuss this new approach in a more general framework. 
By way of examples the cases $A_q(3k, 4; k)$, $k \ge 5$, and $A_q(6k, 2k; 2k)$, $k \ge 4$ even, are considered.

\section{Preliminaries}
\label{sec_preliminaries}     
 
Let $V$ denote an $n$-dimensional vector space over $\F_q$. Since $V\simeq \F_q^n$ induces an isometry $(S(V),d)\simeq (S(\F_q^n),d)$, the particular choice of the ambient 
vector space $V$ does not matter here, so that we will always write $\F_q^n$ in the following. An \emph{$(n, \Lambda, d; k)_q$ constant--dimension code (CDC)} is a set 
$\cal C$ of $k$-dimensional subspaces of $\F_q^n$ with $\# {\cal C} = \Lambda$ and minimum subspace distance $d({\cal C}) = \min\{d(U,U') \,:\, U,U'\in {\cal C}, U \neq U' \} = d$. 
In the terminology of projective geometry, an $(n, \Lambda, 2\delta; k)_q$ constant--dimension code, $\delta>1$, is a set $\cC$ of $(k - 1)$-dimensional projective subspaces 
of $\PG(n - 1, q)$ such that $\# \cC = \Lambda$ and every $(k-\delta)$-dimensional projective subspace of $\PG(n-1, q)$ is contained in at most one member of $\cC$ or, 
equivalently, any two distinct codewords of $\cC$ intersect in at most a $(k -\delta -1)$-dimensional projective space. The maximum size of an $(n, \star, d; k)_q$-CDC is 
denoted by $A_q(n, d;k)$. The number of $k$-dimensional subspaces of $\F_q^n$ is given by $\gaussmnum{n}{k}{q}=\prod_{i=0}^{k-1} \frac{q^{n-i}-1}{q^{k-i}-1}$.

If $U$ is a $k$-dimensional subspace of $\F_q^n$ we write $U\le\F_q^n$ and call $U$ a \emph{$k$-subspace}. The row space $R(M)$ of any full-rank matrix $M\in\F_q^{k\times n}$ 
gives rise to such a $k$-subspace $U$. Here $M$ is called a \emph{generator matrix} of $U$. For the other direction we denote by $\tau(U)$ the unique full-rank matrix in 
$\F_q^{k\times n}$ that is in \emph{reduced row echelon form} (rre). By $p(U)\in\F_2^n$ we denote the binary vector whose $1$-entries coincide with the pivot columns of 
$\tau(U)$. Its \emph{Hamming weight} $\wham(p(U))$, i.e., the number of non-zero entries, equals the dimension $k$ of $U$. Slightly abusing notation, we also write 
$\tau(M)=\tau(R(M))$ and $p(M)=p(R(M))$ for a matrix $M\in\F_q^{k\times n}$. For $M=\begin{pmatrix}1&0&1&0\\1&1&0&1\end{pmatrix}\in\F_2^{2\times 4}$ we have 
$\tau(M)=\begin{pmatrix}1&0&1&0\\0&1&1&1\end{pmatrix}$ and $p(M)=(1,1,0,0)$. The subspace distance $d(U,U')$ between two subspaces $U$ and $U'$ of $\F_q^n$ can be expressed 
via the ranks of their generator matrices:
\begin{eqnarray}
  \label{eq_d_s_rk}
  d(U,U')&=&\dim(U+U')-\dim(U\cap U')=2\dim(U+U')-\dim(U)-dim(U')\nonumber\\ 
  &=&2\rk\!\left(\begin{smallmatrix}\tau(U)\\ \tau(U')\end{smallmatrix}\right)-\rk(\tau(U))-\rk(\tau(U')).
\end{eqnarray}
Note that this equation remains true if we replace $\tau$ by any other normal form of $k$-subspaces in $\F_q^n$ as full-rank $(k\times n)$-matrices over $\F_q$. If $U$ and 
$U'$ have the same dimension, say $k$, then their subspace distance has to be even and is at most $2k$. In the case $d=2k$ a CDC is also known as a \emph{partial $k$-spread} 
and we say that the codewords are pairwise \emph{disjoint}, i.e., they intersect trivially. We speak of a \emph{$k$-spread} if the cardinality
$\gaussmnum{n}{1}{q}/\gaussmnum{k}{1}{q}$ 
is attained, which is possible if and only if $k$ divides $n$, see e.g.~\cite{Andre,Segre}.

If $p(U)=p(U')$, then Equation~(\ref{eq_d_s_rk}) simplifies to $d(U,U')=2\rk(\tau(U)-\tau(U'))$. More generally, for two matrices $M,M'\in\F_q^{m\times n}$ we define the 
\emph{rank distance} via $\drank(M,M')=\rk(M-M')$, so that $\left(\F_q^{m\times n},\drank\right)$ is a metric space. A subset $\mathcal{M}\subseteq \F_q^{m\times n}$ is 
called a \emph{rank metric code}. More precisely, we speak of an $(m\times n,d_r)_q$-rank metric code, where $d_r$ is the minimum rank distance 
$\drank(\mathcal{M})=\min\{\drank(M,M') \; : \; M,M'\in\mathcal{M}, M\neq M'\}$. A rank metric code is called \emph{linear} if it is a subspace of $\F_q^{m\times n}$ 
over $\F_q$ and \emph{additive} if it is closed under addition. The maximum size of an $(m\times n,d_r)_q$-rank metric code is given by 
$m(q,m,n,\drank) := q^{\max\{m,n\}\cdot(\min\{m,n\}-\drank+1)}$. A rank metric code ${\cal M} \subseteq \F_q^{m\times n}$ attaining this bound is said to be a 
{\em maximum rank distance (MRD) code} with parameters $(m \times n, \drank)_q$ or {\em $(m \times n, \drank)_q$--MRD code}, see e.g.\ the recent survey \cite{sheekey2019mrd}. 
Linear MRD codes exist for all parameters. Moreover, for $\drank<\drank'$ we can assume the existence of a linear $(m \times n, \drank)_q$--MRD code that contains an 
$(m \times n, \drank')_q$--MRD code as a subcode. The rank distribution of an additive $(m \times n, \drank)_q$--MRD code is completely determined by its parameters, i.e., 
the number of codewords of rank $r$ is given by 
\begin{equation}
  a(q,m,n,\drank,r)=\gaussmnum{\min\{n,m\}}{r}{q}\sum_{s=0}^{r-d_r} (-1)^sq^{{s\choose 2}}\cdot\gaussmnum{r}{s}{q}\cdot\left(q^{\max\{n,m\}\cdot(r-d_r-s+1)}-1\right)
\end{equation}
for all $d_r\le r\le \min\{n,m\}$, see e.g.\ \cite[Theorem 5.6]{delsarte1978bilinear} or \cite[Theorem 5]{sheekey2019mrd}. Clearly, there is a unique codeword of rank strictly 
smaller than $\drank$ -- the zero matrix.    

The \emph{Hamming distance} $\dham(u,u')=\# \{i \mid u_i \neq u'_i\}$, for two vectors $u,u' \in \mathbb{F}_2^n$, can be used to lower bound the subspace distance between 
two subspaces $U$ and $U'$  (not necessarily of the same dimension) of $\F_q^n$:
\begin{lemma}\cite[Lemma~2]{etzion2009error}
\label{lemma_d_s_d_h}
For 
$U,U'\le \F_q^n$, we have $d(U,U') \ge \dham(p(U),p(U'))$.
\end{lemma} 
 
\section{Known constructions and bounds for constant--dimension codes}
\label{sec_known_construction_and_bounds} 

First we note that for bounds on $A_q(n,d;k)$ it suffices to consider the cases $k\le\tfrac{n}{2}$. Given a non-degenerate bilinear form, we denote by $U^\perp$ the orthogonal 
subspace of a subspace $U$, which then has dimension $n-\dim(U)$. With this, we have $d(U, W) = d(U^\perp, W^\perp)$, so that $A_q(n,d;k)=A_q(n,d;n-k)$ and we assume $k\le\tfrac{n}{2}$ 
in the following. Since each $(k-\tfrac{d}{2}+1)$-subspace is contained in at most one codeword, we have
$A_q(n,d;k)\le \gaussmnum{n}{k-\tfrac{d}{2}+1}{q}/\gaussmnum{k}{k-\tfrac{d}{2}+1}{q}$,  
see \cite[Theorem 5.2]{wang2003linear}. For fixed parameters $d$ and $k$ this bound is asymptotically optimal, see \cite{frankl1985near}. For $d<2k$ the recursive Johnson bound 
$A_q(n,d;k)\le \left\lfloor \gaussmnum{v}{1}{q} \cdot A_q(n-1,d;k-1)/\gaussmnum{k}{1}{q}\right\rfloor$,  
see \cite{XF}, improves upon that. Besides the tightening of this bound, based on divisible codes, see \cite[Theorem 5]{ubt_eref48605}, the only known improvements are 
$A_2(6,4;3)=77<81$ \cite{ubt_eref7922} and $A_2(8,6;4)=257< 289$ \cite{ubt_eref46984}. For partial spreads all known upper bounds can be derived from the non-existence of 
projective divisible codes of a certain length and divisibility, see \cite{ubt_eref42176}. This includes the exact determination of $A_q(n,2k;k)$ for large $k$, 
see \cite{nuastase2017maximum}, as well as several explicit analytical lower bounds, see \cite{ubt_eref36705}. For other known, but weaker, upper bounds for CDCs we refer 
e.g.\ to the survey in \cite{ubt_eref40824}.

With respect to the best known constructions, or lower bounds for $A_q(n,d;k)$, the situation is not that overseeable. Here we only mention few general approaches, that give 
the record codes in the majority of the parameter cases covered in \cite{tables}, and refer e.g.\ to the recent survey \cite{horlemann2018constructions}. Based on 
Lemma~\ref{lemma_d_s_d_h}, in \cite{etzion2009error} the \emph{Echelon-Ferrers construction} was introduced, see e.g.\ \cite{ST} for refinements. Here different subcodes with 
diverse pivot vectors are combined according to Lemma~\ref{lemma_d_s_d_h}. Considering only the pivot vector $(1,\dots,1,0,\dots 0)$ this contains the so-called lifted MRD  
(LMRD) codes from \cite{SKK}. Here, the codewords are of the form $R(I_k|M)$, where $I_k$ denotes the $k\times k$ unit matrix and $M$ is a matrix from an MRD code. This 
construction yields $A_q(n,d;k)\ge m(q,k,n-k,\tfrac{d}{2})$. A bit more general, we can consider codewords of the form $R(\tau(U)|M)$, where $M$ is an element of an 
$(k\times(n-m),\tfrac{d}{2})_q$-MRD code and $U$ is an element of an $(m,d;k)_q$-CDC. Since this \emph{lifting step} created an $(n-m)$-subspace that is disjoint to all 
codewords, more codewords can be added. This approach is called the \emph{linkage construction} \cite{MR3543532}, see also \cite{ST}, and yields 
$A_q(n,d;k)\ge A_q(m,d;k)\cdot m(q,k,n-m,\tfrac{d}{2}) + A_q(n-m,d;k)$. We will generalize this approach in Lemma~\ref{lemma_construction_1}. Finally, for constructions 
obtained by using geometrical techniques we refer the interested reader to \cite{cossidente2019subspace, CP, CP1, CP2, CPS}.   
 
\section{Constructions based on rank metric codes} 
\label{sec_constructions_rank_metric}

In this subsection we aim at constructive lower bounds for $A_q(n,d;k)$ using rank metric codes in different variants. As mentioned before, we assume $2k\le n$. 

\begin{lemma}
  \label{lemma_construction_1}
  For a subspace distance $d$, 
  let $\bar{n}=\left(n_1,\dots,n_l\right)\in\mathbb{N}^l$, where $l \ge 2$, be such that $\sum_{i=1}^l n_i=n$ and $n_i\ge k$ for all $1\le i\le l$. 
  Let $\mathcal{C}_i$ be an $(n_i,\star,d;k)_q$-CDC and $\mathcal{M}_i$ be a $(k\times n_i,\tfrac{d}{2})_q$-rank metric code for $1\le i\le l$. Then 
  $\mathcal{C}=\bigcup_{i=1}^l\mathcal{C}^i$, where
  \begin{eqnarray*}
    \mathcal{C}^i=\Big\{R\!\left(M_1|\dots|M_{i-1}|\tau(U_i)|M_{i+1}|\dots|M_l\right) &:&
    U_i\in\mathcal{C}_i, M_j\in\mathcal{M}_j, \,\forall 1\le j\le l, i\neq j,\\ 
    &&\text{ and } \operatorname{rk}(M_j)\le k-\tfrac{d}{2}, \,\forall 1\le j<i \Big\},
  \end{eqnarray*} 
  is an $(n,\star,d;k)_q$-CDC of cardinality
  $$
    \#\mathcal{C}=\sum_{i=1}^l \left(\prod_{j=1}^{i-1} \#\left\{M\in\mathcal{M}_j\,:\, \rk(M)\le k-\tfrac{d}{2}\right\}\right)\cdot 
    \#\mathcal{C}_i\cdot\left(\prod_{j=i+1}^{l} \#\mathcal{M}_j\right).
  $$
\end{lemma}
\begin{proof}
Since $\rk(\tau(U_i))=k$ for all $U_i\in\mathcal{C}_i$ the elements of $\mathcal{C}^i$ are $k$-subspaces of $\F_q^n$ for all $1\le i\le l$; so the elements of $\mathcal{C}$ 
are $k$-subspaces of $\F_q^n$. 
  
For the distance analysis let $U\in\mathcal{C}^{i}$, $U'\in\mathcal{C}^{i'}$ for some indices $1\le i\le i'\le l$. By construction there exist $U_i\in\mathcal{C}_i$ and 
$M_j\in\mathcal{M}_j$ for   $1\le j\le l$, $j\neq i$, with $$U=R(M_1|\dots|M_{i-1}|\tau(U_i)|M_{i+1}|\dots| M_l)$$ and $\rk(M_j)\le k-\tfrac{d}{2}$ for all $1\le j<i$. 
Similarly, there exist $U'_{i'}\in\mathcal{C}_{i'}$ and $M'_j\in\mathcal{M}_j$ for $1\le j\le l$, $j\neq i'$, with 
$U'=R(M'_1|\dots|M'_{i'-1}|\tau(U'_{i'})|M'_{i'+1}|\dots| M'_l)$ and $\rk(M'_j)\le k-\tfrac{d}{2}$ for all $1\le j<i'$. 
 
If $i<i'$ we set $\overline{U}=R(\tau(U_i)|M_{i'})$ and $\overline{U}'=R(M'_i|\tau(U'_{i'}))$, which are both $k$-subspaces of $V\simeq\F_q^{n_i+n_{i'}}$ and satisfy 
$d(U,U')\ge d(\overline{U},\overline{U}')\ge d$. The later inequality follows from Lemma~\ref{lemma_d_s_d_h} and $\dham(p(\overline{U}),p(\overline{U}'))\ge d$, which 
is true since $p(\overline{U})$ has its $k$ ones in the first $n_i$ components while $p(\overline{U}')$ has at least $\tfrac{d}{2}$ of its $k$ ones in the last $n_{i'}$ 
components.
  
If $i=i'$ and $U_i\neq U'_i$ we have $d(U,U')\ge d(U_i,U'_i)\ge d$ since $U_i,U'_i\in\mathcal{C}_i$ and $d(\mathcal{C}_i)\ge d$. Now let $i=i'$, $U_i=U'_i$, and $1\le j\le l$ 
be an index with $M_j\neq M'_j$ and $j\neq i$. For $\overline{U}=R(\tau(U_i)|M_j)$,  $\overline{U}'=R(\tau(U_i),M'_j)$, we have 
$d(U,U')\ge d(\overline{U},\overline{U}')$ and $d(\overline{U},\overline{U}')\ge d$, since $M_j\neq M'_j\in\mathcal{M}_j$ and $\drank(\mathcal{M}_j)\ge \tfrac{d}{2}$.
\end{proof}

We remark that Lemma~\ref{lemma_construction_1} generalizes \cite[Theorem 3]{XC}, \cite[Theorem 3.1]{chen2019new}, \cite[Theorem 4.1]{chen2019new},  \cite[Theorem 1]{he2019construction}, 
and \cite[Theorem 3]{HC}.\footnote{We remark that in Lemma~\ref{lemma_construction_1} we can increase the dimension of the ambient 
space of the CDC $\mathcal{C}_i$ if we further restrict the possible ranks of the elements in $\mathcal{M}_j$ for $1\le j<i$. For details for the 
special case $l=2$ see \cite[Theorem 24]{DH}, which e.g.\ allows to also treat the improved linkage construction from \cite{ubt_eref40824} in that 
framework. Also the subsequent results in Lemma~\ref{lemma_special_substructure} and Lemma~\ref{lemma_construction_2} can be adjusted to the end of that modification. However, 
as we are not aware of any specific parameters, where this approach leads to a strict improvement over a known code, we refrain from discussing the 
details.} 
Rank-metric codes of constant rank with a lower bound on the minimum rank-distance have been studied in 
\cite{gadouleau2010constant} and generalized in \cite{DH,LCF}. Here we restrict ourselves on subcodes contained in additive MRD codes.
 
\begin{cor}
  \label{cor_construction_1}
  For a subspace distance $d$, 
  $\bar{n}=\left(n_1,\dots,n_l\right)\in\mathbb{N}^l$, $l\ge 2$, be such that $\sum_{i=1}^l n_i=n$ and $n_i\ge k$ for all $1\le i\le l$. 
  Then, we have
  $$
    A_q(n,d;k)\ge \sum_{i=1}^l \left(\prod_{j=1}^{i-1} \left(1+\sum_{r=\tfrac{d}{2}}^{k-\tfrac{d}{2}}a(q,k,n_j,\tfrac{d}{2},r)\right)\right)\cdot 
    A_q(n_i,d;k)\cdot\left(\prod_{j=i+1}^{l}m(q,k,n_j,\tfrac{d}{2})\right).
  $$
\end{cor} 

Note that $\bar{n}$ also specifies $l$. While we have no restriction on $d$ in principle, $d>k$ forces $1+\sum_{r=\tfrac{d}{2}}^{k-\tfrac{d}{2}}a(q,k,n_j,\tfrac{d}{2},r)=1$. 
If we apply Corollary~\ref{cor_construction_1} with $\bar{n}=(4,4,4)$ and use $A_2(4,4;4)=1$,   
we obtain $A_2(12,4;4)\ge 19\,208\,388$. With $\bar{n}=(8,4)$, 
$A_2(8,4;4)\ge 4801$ \cite{BOW}, and $A_2(4,4;4)=1$, we obtain 
\begin{equation}\label{construction1_A_2(12,4;4)}
A_2(12,4;4)\ge 19\,673\,822.
\end{equation} 
Aside from very recent preprints, the previously best known lower bound was $A_2(12,4;4)\ge 19\,664\,917$, obtained from the improved linkage construction \cite{ubt_eref40824}. 
Moreover, the constant--dimen\-sion codes from Lemma~\ref{lemma_construction_1} have some special structure that allows to add more codewords.

\begin{lemma}
  \label{lemma_special_substructure} 
With the same notation used in Lemma~\ref{lemma_construction_1}, set $\sigma_i=\sum_{j=1}^{i} n_j$, $1 \le i \le l$ and $\sigma_0 = 0$. Let $E_i$ denote the $(n-n_i)$-subspace of 
$\F_q^n$ consisting of all vectors in $\F_q^n$ that have zeroes for the coordinates between $\sigma_{i-1}+1$ and $\sigma_i$ for all $1\le i\le l$. Then, the elements of 
$\mathcal{C}^i$ are disjoint from $E_i$ for all $1\le i\le l$. 
\end{lemma} 
\begin{proof}
Let $U\in \mathcal{C}^i$ be arbitrary. By construction there exist $U_i\in\mathcal{C}_i$ and $M_j\in\mathcal{M}_j$ for   $1\le j\le l$, $j\neq i$ with $U=R(M)$, where 
$M=(M_1|\dots|M_{i-1}|\tau(U_i)|M_{i+1}|\dots| M_l)$, and $\rk(M_j)\le k-\tfrac{d}{2}$ for all $1\le j<i$. Note that $E_i = R(N)$ and $\tau(E_i) = N$, where 
$N \in \F_q^{(n - n_i) \times n}$ is obtained from the unit matrix $I_n$ by deleting the rows in position between $\sigma_{i-1}+1$ and $\sigma_i$. 
Consider a non-trivial linear combination  
of the $k$ rows of $M$. The entries in the coordinates between $\sigma_{i-1}+1$ and $\sigma_i$ are obtained by the same non-trivial linear combination applied  
to $\tau(U_i)$. Since $\rk(\tau(U_i))=k$ the statement follows.
\end{proof}

In our next construction we want to combine several CDCs in the same ambient space. In order to express that every codeword from a CDC $\mathcal{C}$ has a subspace 
distance of at least $d$ to any codeword from another CDC $\mathcal{C}'$ we write $d(\mathcal{C},\mathcal{C}')\ge d$. 

\begin{lemma}
  \label{lemma_construction_2}
Let $\mathcal{C}$ be a subspace code as in Lemma~\ref{lemma_construction_1} with corresponding $\bar{n}\in\mathbb{N}^l$, $\bar{a}=\left(a_1,\dots,a_l\right)\in \mathbb{N}^l$ 
and $\bar{b}=\left(b_1,\dots,b_l\right)\in\mathbb{N}^l$ with $\sum_{i=1}^l a_i=k$, $\sum_{i=1}^l b_i=k-\tfrac{d}{2}$, and $\tfrac{d}{2}\le a_i, 
b_i<a_i\le n_i$, for all $1\le i \le l$. For an integer $s$, let $\mathcal{D}_i^j$ be an $(n_i, \star, d; a_i)_q$-CDC, for all $1\le i\le l$ and all $1\le j\le s$, such that 
$d(\mathcal{D}_i^{j_1},\mathcal{D}_i^{j_2})\ge 2a_i-2b_i$, for all $1\le i\le l$ and all $1\le j_1<j_2\le s$. Then, there exists an $(n, \star, d; k)_q$-CDC, say $\mathcal{D}$, 
with cardinality
$$
    \#\mathcal{D}=\sum_{j=1}^s \prod_{i=1}^l \#\mathcal{D}_i^j ,
$$
such that $\mathcal{C} \cap \mathcal{D}=\emptyset$ and $\mathcal{C}\cup\mathcal{D}$ is also an $(n, \star, d; k)_q$-CDC.
\end{lemma} 
\begin{proof} 
Let $\sigma_i=\sum_{h=1}^i n_h$ for all $1\le i\le l$ and $\sigma_0=0$. For each $1\le i\le l$ and $1\le j\le s$ let $\overline{\mathcal{D}}_i^j$ be an embedding of 
$\mathcal{D}_i^j$ in $\F_q^n$ such that the vectors contained in an element of $\overline{\mathcal{D}}_i^j$ have non-zero entries only in the coordinates between $\sigma_{i-1}+1$ 
and $\sigma_i$. With this  we set  
   $$
    \mathcal{D}=\bigcup_{j=1}^s \left\{U_1 \times U_2 \times \dots \times U_l \,:\, U_h \in \overline{\mathcal{D}}_h^j, \,\, \forall 1 \le h \le l \right\},
   $$
where $U_1 \times U_2 \times \dots \times U_l$ denotes the smallest subspace that contains $U_1,\dots, U_l$, i.e., the span of these subspaces.    
Since $\sum_{i=1}^l a_i=k$ and $U_{h_1}, U_{h_2}$, $1 \le h_1, h_2 \le l$, $h_1 \ne h_2$, are disjoint, then the elements of $\mathcal{D}$ are $k$-subspaces of $\F_q^n$.
   
Next we want to show that $d(W,U)\ge d$, for all $W\in\mathcal{C}$ and for all $U\in \mathcal{D}$. With the same notation used in Lemma~\ref{lemma_construction_1}, there exists 
an index $1\le i_0 \le l$ with $W\in\mathcal{C}^{i_0}$, so that Lemma~\ref{lemma_special_substructure} gives that $\dim(W\cap E_{i_0})=0$. 
Since $\dim(U \cap E_{i_0}) = \sum_{i = 1, \; i \ne i_0}^{l} a_i \ge (l-1) \tfrac{d}{2} \ge \tfrac{d}{2}$, we have that $\dim(W \cap U) \le k-\tfrac{d}{2}$ and $d(W,U)\ge d$. 
   
Now let $U,U'\in\mathcal{D}$, $U \ne U'$, with $U=U_1\times\dots\times U_l$ and $U'=U'_1\times\dots\times U'_l$ for $U_i\in\overline{\mathcal{D}}_i^j$ and 
$U'_i\in\overline{\mathcal{D}}_i^{j'}$, where $1\le i\le l$ and $1\le j, j'\le s$. If $j=j'$ then there exists an index $1\le i^*\le l$ with $U_{i^*}\neq U'_{i^*}$. Since 
$U_{i^*}, U'_{i^*}\in\overline{\mathcal{D}}_{i^*}^j$, we have $d(U_{i^*},U'_{i^*})\ge d$, so that $\dim(U_{i^*} \cap U'_{i^*}) \le a_i - \tfrac{d}{2}$, which implies 
$\dim(U\cap U')\le k-\tfrac{d}{2}$ and $d(U,U')\ge d$. It remains to consider the case $j\neq j'$, where we have $\dim(U_i\cap U'_i)\le b_i$ for all $1\le i\le l$. Thus, 
$\dim(U\cap U')\le \sum_{i=1}^l b_i=k-\frac{d}{2}$ and $d(U,U')\ge d$.
   
The formula for the cardinality of $\mathcal{D}$ is obvious from the construction.
\end{proof}
 
Note that $a_i\ge\tfrac{d}{2}$ and $\sum_{i=1}^l a_i=k$ imply $k\ge\tfrac{ld}{2}\ge d$, i.e., the construction of Lemma~\ref{lemma_construction_2} works for small subspace 
distances $d$ only. 

In what follows we apply Lemma~\ref{lemma_construction_2} 
in order to achieve a lower bound for $A_q(8, 4; 4)$ and $A_q(12, 4; 4)$. In the former case we obtain the best known lower bound if $q >2$ \cite{CP2,ES}, whereas in the 
latter case we get an improvement for any $q$. We need the following definition. A \emph{$k$-parallelism} of $\F_q^{kt}$ is a set of 
$\gaussmnum{kt}{k}{q}\cdot\gaussmnum{k}{1}{q}/\gaussmnum{kt}{1}{q}$   
pairwise disjoint $k$-spreads of $\F_q^{kt}$. For each $q$ a $2$-parallelism of $\F_q^4$ exists, see \cite{beutelspacher1974parallelisms,denniston1972some}.

Let ${\bar n} = (4, 4)$, $\bar{a}=(2,2)$ and $\bar{b}=(1,1)$. Lemma~\ref{lemma_construction_1} gives a CDC with $q^{12} + (q^2-1)(q^2+1)^2(q^2+q+1)+1$ codewords. To apply 
Lemma~\ref{lemma_construction_2} we can choose $\mathcal{D}_i^j$, $i = 1,2$, as a $2$-spread of $\F_q^4$ such that $s := q^2+q+1$ and $\{\cD_i^j \; : \; 1 \le j \le s\}$ is a 
$2$-parallelism of $\F_q^4$, $i = 1,2$. Here the CDC $\mathcal{D}_i^j$ matches the upper bound $A_q(4, 4; 2) = q^2+1$.  
It follows that 
\begin{equation}\label{construction_A_q(8,4;4)}
A_q(8,4;4)\ge q^{12} + q^2(q^2 + 1)^2(q^2 + q + 1) + 1.
\end{equation}

In order to apply Lemma~\ref{lemma_construction_2} for $A_q(12, 4; 4)$, we can choose $\bar{n}=(8,4)$, $\bar{a}=(2,2)$, and $\bar{b}=(1,1)$. Similarly to the previous case, we 
can choose $\mathcal{D}_2^j$ as a $2$-spread of $\F_q^4$ such that $s := q^2+q+1$ and $\{\cD_2^j \; : \; 1 \le j \le s\}$ is a $2$-parallelism of $\F_q^4$. We can define 
$\mathcal{D}_1^j$ as a $2$-spread of $\F_q^8$ such that $\{\cD_1^j \; : \; 1 \le j \le s\}$ is a collection of $s$ pairwise disjoint $2$-spreads of $\F_q^8$. In order to do so 
let $\mathcal{S}$ be a $4$-spread of $\F_q^8$, i.e., $\#\mathcal{S}=q^4+1$. For each $4$-subspace $S \in \mathcal{S}$ we replace $S$ with a $2$-parallelism of $S$. This results 
in $\#\mathcal{D} = (q^2+q+1) (q^2+1)^2 (q^4+1)$. 
Taking into account the lower bound for $A_q(8,4;4)$ obtained above, the previous discussion together with Lemma~\ref{lemma_construction_1} and Corollary~\ref{cor_construction_1}, 
gives rise to a $(12,\star,4;4)_q$-CDC $\mathcal{C}$ with cardinality 
\begin{eqnarray*}
 \#\mathcal{C} & = & q^{12}\left(q^{12} + q^2(q^2 + 1)^2(q^2 + q + 1) + 1\right)+(q + 1)(q^2 + 1)^2(q - 1)(q^2 + q + 1)(q^4 + 1) + 1\\
  & = & q^{24} \!+\!q^{20}\!+\!q^{19}\!+\!3q^{18}\!+\!2q^{17}\!+\!3q^{16}\!+\!q^{15}\!+\!q^{14}\!+\!2q^{12}\!+\!q^{11}\!+\!2q^{10}\!+\!q^9\!+\!q^8\!-\!q^4\!-\!q^3\!-\!2q^2\!-\!q .
\end{eqnarray*}
For $q\ge 3$ the previously best known lower bound was $A_q(12,4;4)\ge q^{24}+q^{20}+q^{19}+3q^{18}+2q^{17}+3q^{16}+q^{15}+q^{14}+q^{12}+q^{10}+2q^8+2q^6+2q^4 +q^2$, see 
\cite[Proposition 4.6]{K}. Something more can be said in the case when $q = 2$, indeed by combining the previous argument together with \eqref{construction1_A_2(12,4;4)} we obtain 
$$
  A_2(12,4;4) \ge 19\,676\,797, 
$$
which strictly improves upon the corresponding results in \cite{chen2019new,he2019construction,HC,DH,XC}. 

We will further improve the lower bound for $A_q(12,4;4)$, $q \ge 3$, in Section~\ref{sec_combine}.       

We remark that for each $1\le i\le l$, the CDC $\bigcup_{j=1}^s \mathcal{D}_i^j$ is an $(n_i,\star,2a_i-2b_i;a_i)_q$-code. Partitioning it into subcodes with 
subspace distance $d>2a_i-2b_i$ is a hard problem in general and was e.g.\ considered in the context of the \emph{coset construction} for CDCs, see \cite{ubt_eref40356}. 
If we restrict ourselves to LMRD codes, then one can determine an analytical lower bound.
  
\begin{cor}
  \label{cor_construction_2}
  In Lemma~\ref{lemma_construction_2} one can achieve
  $$
    \#\mathcal{D}\ge \min\{\alpha_i\,:\,1\le i\le l\}\cdot\prod_{i=1}^l m\!\left(q,a_i,n_i-a_i,\tfrac{d}{2}\right),
  $$
  where $\alpha_i=m\!\left(q,a_i,n_i-a_i,a_i-b_i\right)/m\!\left(q,a_i,n_i-a_i,\tfrac{d}{2}\right)$.
\end{cor}
\begin{proof}
For $1\le i\le l$ let $\mathcal{M}_i$ be a linear $(a_i\times(n_i-a_i),a_i-b_i)_q$--MRD code that contains a linear $(a_i\times(n_i-a_i),\tfrac{d}{2})_q$--MRD code 
$\overline{\mathcal{M}}_i$ as a subcode. For each $M\in\mathcal{M}_i$ we can consider $M+\overline{\mathcal{M}}_i=\left\{M+\overline{M}\,:\,\overline{M}\in \overline{\mathcal{M}}_i\right\}$. 
Note that $M+\overline{\mathcal{M}}_i$ is an $(a_i\times(n_i-a_i),\tfrac{d}{2})_q$--MRD code. 
Moreover, we have $M+\overline{\mathcal{M}}_i=M'+\overline{\mathcal{M}}_i$ if and only if $M-M'\in \overline{\mathcal{M}}_i$ and $\left(M+\overline{\mathcal{M}}_i\right)
\cap\left(M'+\overline{\mathcal{M}}_i\right) =\emptyset$ otherwise. In other words, we consider the $\alpha_i$ cosets of $\overline{\mathcal{M}}_i$ in $\mathcal{M}_i$. With this, 
$\mathcal{M}_i$ can be partitioned into $\alpha_i$ maximum rank distance codes with parameters $(a_i\times(n_i-a_i),\tfrac{d}{2})_q$; hence each of them has cardinality 
$m\!\left(q,a_i,n_i-a_i,\tfrac{d}{2}\right)$. By lifting with an $(a_i\times a_i)$-unit matrix we obtain the CDCs $\mathcal{D}_i^j$ with the required properties of 
Lemma~\ref{lemma_construction_2} for $s=\min\{\alpha_i\,:\,1\le i\le l\}$. 
\end{proof}  


In order to obtain a good lower bound for $A_q(n,d;k)$ we can combine Corollary~\ref{cor_construction_1} with Lemma~\ref{lemma_construction_2} and Corollary~\ref{cor_construction_2}. 
As an example we consider $A_q(12,6;6)$. For $\bar{n}=(6,6)$ Corollary~\ref{cor_construction_1} gives a $(12,\star,6;6)_q$-CDC $\mathcal{C}$ with cardinality 
\begin{eqnarray*}
  \#\mathcal{C}&=& q^{24}+(q^2 + 1)(q^5 - 1)(q^5 + q^4 + q^3 + q^2 + q + 1)(q^3 + 1) \\ 
  &=& q^{24}\!+\!q^{15}\!+\!q^{14}\!+\!2q^{13}\!+\!3q^{12}\!+\!3q^{11}\!+\!3q^{10}\!+\!2q^9\!+\!q^8\!-\!q^7\!-\!2q^6\!-\!3q^5\!-\!3q^4\!-\!3q^3\!-\!2q^2\!-\!q\!-\!1.
\end{eqnarray*} 
Actually this matches the best known lower bound for $A_q(12,6;6)$ for all field sizes $q$, see \cite[Theorem 3]{XC}. By using Lemma~\ref{lemma_construction_2} via 
Corollary~\ref{cor_construction_2} with $\bar{a}=(3,3)$ and $\bar{b}=(1,2)$ we get $\#\cD_i^j = q^3$, $1 \le j \le q^3$, $i = 1,2$, and hence $q^3\cdot q^3\cdot q^3=q^9$ 
additional codewords. 
In some cases, the CDC so obtained can be enlarged. Indeed, let $E_1$ or $E_2$ denote the $6$-subspace of $\F_q^{12}$ consisting of all vectors in $\F_q^{12}$ that have 
zeroes in the first or in the last six coordinates, respectively. Let $\overline{\mathcal{D}}_i^j$ be an embedding of $\mathcal{D}_i^j$ in $\F_q^{12}$ such that the vectors 
contained in an element of $\overline{\mathcal{D}}_i^j$ are in $E_i$, $i = 1,2$. Note that, by construction, there exists a special $3$--subspace of $E_i$, say $\xi_i$, such 
that every member of $\overline{\mathcal{D}}_i^j$ is disjoint from $\xi_i$, $i = 1,2$. Let $\cF_1$ be the set consisting of the $6$-subspaces of $\F_q^{12}$ spanned by $\xi_1$ 
and a member of $\overline{\mathcal{D}}_2^1$. Similarly, let $\cF_2$ be the set consisting of the $6$-subspaces of $\F_q^{12}$ spanned by $\xi_2$ and a member of 
$\overline{\mathcal{D}}_1^1$. Then it can be easily checked that the CDC constructed above can be enlarged by adding the $2q^3$ codewords of $\cF_1 \cup \cF_2$. This leads to 
$$
A_q(12, 6; 6) \ge q^{24}\!+\!q^{15}\!+\!q^{14}\!+\!2q^{13}\!+\!3q^{12}\!+\!3q^{11}\!+\!3q^{10}\!+\!3q^9\!+\!q^8\!-\!q^7\!-\!2q^6\!-\!3q^5\!-\!3q^4\!-\!q^3\!-\!2q^2\!-\!q\!-\!1 .
$$


Let $k \ge 4$ be a positive even integer and consider a CDC with parameters $(4k,2k;2k)_q$ obtained from Corollary~\ref{cor_construction_1} with $\bar{n}=(2k,2k)$, 
Corollary~\ref{cor_construction_2} with $\bar{a}=(k,k)$, $\bar{b}=\left(\tfrac{k}{2},\tfrac{k}{2}\right)$. 
A similar argument to that used in the previous paragraph shows that the CDC so obtained can be enlarged by adding a further $2q^k$ codewords. Then
  \begin{eqnarray} \label{construction_A_q(4k, 2k; 2k)}
   A_q(4k, 2k; 2k) & \ge & m(q,2k,2k,k) +a(q,2k,2k,k,k) + m(q,k,k,k)\cdot m(q,k,k,\tfrac{k}{2}) + 2q^k + 1 \notag \\
                  & = & q^{2k(k+1)} + a(q,2k,2k,k,k) + q^{k(k/2+2)} + 2q^k + 1.
  \end{eqnarray}

\subsection{On constant--dimension codes with $d > k$}

The drawback of the 
construction of Lemma~\ref{lemma_construction_2} is that 
it is 
applicable for $d\le k$ only. This is due to the {\lq\lq}product-type{\rq\rq} constructions where the elements of two (or more) codes are combined in all different ways. If we 
use a one-to-one correspondence for the combinations we can construct CDCs for $d>k$:

\begin{lemma}
  \label{lemma_construction_3}
  Let $\bar{n}=(n_1,n_2)\in\mathbb{N}^2$ with $n_1+n_2=n$, $\bar{a}=(a_1, a_2)\in\mathbb{N}^2$ with $a_1+a_2=k$, $a_1\le k-\tfrac{d}{2}$, and $\bar{b}=(b_1,b_2)\in\mathbb{N}^2$ 
  with $b_1+b_2=k-\tfrac{d}{2}$. Let $\mathcal{C}_0$ be an $(n_1,\star,d;k)_q$-CDC, $\mathcal{C}_1$ be an $(n_1,\star,2a_1-2b_1;a_1)_q$-CDC, and $\mathcal{C}_2$ be an 
  $(n_2,\star,2a_2-2b_2;a_2)_q$-CDC. Then there exists an $(n,\star,d;k)_q$-CDC with cardinality $\#C_0\cdot m(q,k,n_2,\tfrac{d}{2})+\min\left\{\#\mathcal{C}_1,\#\mathcal{C}_2\right\}$.   
\end{lemma}
\begin{proof}
  W.l.o.g.\ we assume that $\mathcal{C}_1$ and $\mathcal{C}_2$ have the same cardinality. Let $\overline{\mathcal{C}}_1$ be  an embedding of $\mathcal{C}_1$ in $\F_q^n$ such that the last 
  $n_2$ entries of the vectors contained in the codewords are always zero. Similarly, let $\overline{\mathcal{C}}_2$ be  an embedding of $\mathcal{C}_2$ in $\F_q^n$ such that the first 
  $n_1$ entries of the vectors contained in the codewords are always zero. We choose an arbitrary numbering $U_1^1,\dots,U_1^s$ and $U_2^1,\dots,U_2^s$ of the elements of 
  $\overline{\mathcal{C}}_1$ and $\overline{\mathcal{C}}_2$, respectively, where $s=\#\mathcal{C}_1=\#\mathcal{C}_2$. Let $\mathcal{M}_0$ be a $(k\times n_2,\tfrac{d}{2})_q$--MRD code. 
  With this we set
  $$
    \mathcal{C}=\left\{R(\tau(U)|M)\,:\, U\in\mathcal{C}_0, M\in\mathcal{M}_0\right\}\cup \left\{U_1^i\times U_2^i\,:\, 1\le i\le s\right\}.
  $$
  Obviously the elements of $\mathcal{C}$ are $k$-subspaces of $\F_q^n$ and we have $\#\mathcal{C}=\#C_0\cdot m(q,k,n_2,\tfrac{d}{2})+\min\left\{\#\mathcal{C}_1,\#\mathcal{C}_2\right\}$. 
  
  For the distance analysis let $W,W'\in\mathcal{C}$ be arbitrary. If $W$ and $W'$ are both of the form $R(\tau(U)|M)$, then $d(W,W')\ge d$ follows as in the proof of 
  Lemma~\ref{lemma_construction_1} (or as in the literature on lifting and the linkage construction). If only $W$ is of the first type, then $p(W)$ contains its $k$ ones in 
  the first $n_1$ coordinates, while $p(W')$ contain only $a_1$ of its $k$ ones in the first coordinates. So using $a_1\le k-\tfrac{d}{2}$ Lemma~\ref{lemma_d_s_d_h} gives 
  $d(W,W')\ge |k-a_1|+|a_2|=2(k-a_1)\ge d$. If $W=U_1^i\times U_2^i$ and $W'=U_1^{i'}\times U_2^{i'}$ for $1\le i<i'\le s$, then $\dim(W\cap W')=\dim(U_1^i\cap U_1^{i'})+ 
  \dim(U_2^i\cap U_2^{i'})\le b_1+b_2=k-\tfrac{d}{2}$, so that $d(W,W')\ge d$. 
\end{proof} 
  
\begin{cor}
  \label{cor_construction_3}
  Let $\bar{n}=(n_1,n_2)\in\mathbb{N}^2$ with $n_1+n_2=n$, $\bar{a}=(a_1, a_2)\in\mathbb{N}^2$ with $a_1+a_2=k$, $a_1\le k-\tfrac{d}{2}$, and $\bar{b}=(b_1,b_2)\in\mathbb{N}^2$ 
  with $b_1+b_2=k-\tfrac{d}{2}$. Then
  $$
    A_q(n,d;k)\ge A_q(n_1,d;k)\cdot m(q,k,n_1,\tfrac{d}{2})+\min\!\left\{A_q(n_1,2a_1-2b_1;a_1),A_q(n_2,2a_2-2b_2;a_2)\right\}.
  $$
\end{cor}  

As an example we consider a lower bound for $A_q(12,8;6)$ and apply Corollary~\ref{cor_construction_3} with $\bar{n}=(6,6)$, $\bar{a}=(2,4)$, and $\bar{b}=(0,2)$. Since 
$A_q(6,8;6)=1$ and $A_q(6,4;2)=A_q(6,4;4)=q^4+q^2+1$ we have $A_q(12,8;6)\ge q^{18}+q^4+q^2+1$. The same lower bound can also be obtained using the optimal code within 
the class of Echelon-Ferrers constructions, see \cite{etzion2009error,tables}. Within the $(12,\star,8;6)_q$-CDCs that contain an LMRD the construction is indeed optimal. 
For $A_2(16,12;8)$ we can apply Corollary~\ref{cor_construction_3} with $\bar{n}=(8,8)$, $\bar{a}=(2,6)$, and $\bar{b}=(0,2)$.
 
Another possible approach is to start with the same lifting $\left\{R(\tau(U)|M)\,:\, U\in\mathcal{C}_0, M\in\mathcal{M}_0\right\}$ where $\bar{n}=(n_1,n_2)$. As observed 
in \cite{K} we can add all codewords from an $(n,\star,d;k)_q$ code $\mathcal{C}'$, without decreasing the minimum subspace distance, if all elements of $\mathcal{C}'$ 
intersect an arbitrary but fixed $n_2$-subspace $S$ in dimension at least $\tfrac{d}{2}$ (as it is the case in Lemma~\ref{lemma_construction_3}). We can start with 
an $(n_2,\star,2d-2k;\tfrac{d}{2})_q$-CDC and then step by step enlarge the dimension of the codewords without creating intersections of codewords with a 
dimension strictly larger than $k-\tfrac{d}{2}$. For the special case $d=2k-2$ it was shown in \cite[Theorem 4.2]{K} that $\#\mathcal{C}=A_q(n_2,2d-2k;\tfrac{d}{2})$ 
can indeed be attained. Whether this is possible for $d\le 2k-4$ is an interesting open problem. As a stimulation for further research in this direction we pose an 
explicit open problem:
\begin{prob}
Do there exist $\gaussmnum{5}{3}{q}=\gaussmnum{5}{2}{q}=(q^4 + q^3 + q^2 + q + 1)(q^2 + 1)$ $5$-subspaces of $\F_q^{10}$ pairwise intersecting in dimension at most $2$ 
such that all elements intersect a special fixed $5$-subspace in dimension $3$?
\end{prob}   
If true this would improve the best know lower bound for $A_q(10,6;5)$ for $q\ge 3$ and indeed match the upper bound within the class of such codes containing an LMRD subcode, 
see \cite{ES}. For $q=2$ such a code was found by computer search, see \cite{MRDBoundHeinlein}. 

\section{Duplicating CDCs in several subspaces of a large-dimensional CDC}
\label{sec_combine}

In \cite{cossidente2019subspace} the authors combined several $(6,\star,4;3)_q$-CDCs to show $A_q(9,4;3)\ge q^{12} + 2q^8 + 2q^7 + q^6 + q^5 + q^4 +1$, which improves the previously best known lower bound $A_q(9,4;3)\ge q^{12}+ 2q^8+ 2q^7+q^6+ 1$ obtained from the improved linkage construction, see \cite{ubt_eref40824}. In \cite{K1} the bound was further improved to $A_q(9, 4; 3) \ge q^{12} + 2q^{8} + 2q^{7} + q^6 + 2q^5 + 2q^4 - 2q^2 - 2q + 1$. Here we want to generalize the approach of \cite{cossidente2019subspace, K1} and apply it to a much wider range of parameters.  

\begin{defin}
  \label{def_ndk_sequence}
An \emph{$(n,d,k)$-sequence} of CDCs is a list $\left(\mathcal{D}_0,\dots,\mathcal{D}_r\right)$ of $(n,\star,d;k)_q$-CDCs such that for each index $0\le i\le r$ there exists 
a codeword $U\in\mathcal{D}_i$ and a disjoint $(n-k)$-subspace $S$ such that $\dim(U'\cap S)\le i$ for all $U'\in\mathcal{D}_i$, where $r=k-\tfrac{d}{2}$.
\end{defin} 

We remark that an LMRD code gives an example for $\mathcal{D}_0$ and for $\mathcal{D}_i$, with $i \ge 1$, we can take $\mathcal{D}_0$. Another possibility is to start with 
an arbitrary $(n,\star,d;k)_q$-CDC, pick the special subspace $S$, and remove all codewords whose dimension of the intersection with $S$ is too large. 
  
\begin{defin}
  \label{def_distance_partition}
  A list $\left(\mathcal{C}_0,\dots,\mathcal{C}_r\right)$ is called a \emph{distance-partition} of an $(n,\star,d;k)_q$-CDC $\mathcal{C}$, where $r=k-\tfrac{d}{2}$, if 
  $\mathcal{C}_0,\dots,\mathcal{C}_r$ is a partition of $\mathcal{C}$ and $\bigcup_{j=0}^i\mathcal{C}_j$ is an $(n,\star,2k-2i;k)_q$-CDC for all $0\le i\le r$. 
\end{defin}  
A trivial distance-partition of an $(n,\star,d;k)_q$-CDC $\mathcal{C}$ is given by $(\emptyset,\dots,\emptyset,\mathcal{C})$. A subcode $\mathcal{C}'\subseteq \mathcal{C}$ with 
maximal subspace distance $d=2k$ is called a \emph{partial-spread subcode}. Given such a partial-spread subcode $\mathcal{C}'$, if $d<2k$, then 
$(\mathcal{C}',\emptyset,\dots,\emptyset, \mathcal{C}\backslash\mathcal{C}')$ is a distance-partition of $\mathcal{C}$.

\begin{lemma}
  \label{lemma_construction_4}
Let $\left(\mathcal{C}_0,\dots,\mathcal{C}_r\right)$ be a distance-partition of a $(k+t,\star,d;k)_q$-CDC $\mathcal{C}$ and $\left(\mathcal{D}_0,\dots,\mathcal{D}_r\right)$ be 
a $(k+s,d,k)$-sequence, where $r=k-\tfrac{d}{2}$. If $\mathcal{A}$ is an $(s,\star,d;k)_q$-CDC, then there exists a $(k+s+t,\star,d;k)_q$-CDC $\mathcal{C}'$ with cardinality
  $$
    \#\mathcal{C}'=\#\mathcal{A}+\sum_{i=0}^r \#\mathcal{C}_i\cdot\#\mathcal{D}_{r-i}.
  $$
\end{lemma}
\begin{proof}
In order to build up $\mathcal{C}'$ step by step we embed $\mathcal{C}$ in a $(k+t)$-subspace of $\F_q^{k+s+t}$ and let $S$ be an $s$-subspace of $\F_q^{k+s+t}$ disjoint from it. 
For each codeword $U\in\mathcal{C}$ let $0\le i\le r$, be the index such that $U\in\mathcal{C}_i$. With this, we embed an isomorphic copy $\mathcal{C}_U$ of $\mathcal{D}_{r-i}$ 
in the $(k+s)$-subspace $\langle U,S\rangle$ such that $U$ is a codeword, $S$ the special subspace, and add all those codewords to $\mathcal{C}'$. Let $\overline{\mathcal{A}}$ 
denote the CDC obtained by embedding the codewords of $\mathcal{A}$ in $S$. As a last step add the codewords of $\overline{\cA}$ to $\cC'$. Such a procedure gives rise to 
a $(k+s+t,\star,?;k)_q$-CDC $\mathcal{C}'$ with the stated cardinality.  It remains to check the minimum subspace distance.
  
For $W,W'\in\mathcal{C}' \backslash \overline{\mathcal{A}}$ there exist unique $U,U'\in\mathcal{C}$ such that $W\in\mathcal{C}_U$ and $W'\in\mathcal{C}_{U'}$. Moreover, 
there exist unique indices $0\le i,i'\le r$ with $U\in\mathcal{C}_i$ and $U'\in\mathcal{C}_{i'}$. If $i=i'$ and $W\neq W'$, then $d(W,W')\ge d(\mathcal{D}_{r-i})\ge d$. 
If $i\neq i'$, then w.l.o.g. we assume $i'<i$, so that $\dim(U\cap U')\le i$. By construction we have $\dim(W\cap W')\le\dim(U\cap U')+\dim(W\cap W'\cap S)\le i+\dim(W\cap S)\le r$, 
so that $d(W,W')\ge d$. If $W,W'$ are both contained in $\overline{\mathcal{A}}$, then clearly $d(W,W')\ge d(\mathcal{A})\ge d$. Finally, if $W \in \overline{\cA}$ and 
$W' \in \cC' \backslash \overline{\mathcal{A}}$, then $\dim(W' \cap S) \le k - \tfrac{d}{2}$ and hence $d(W, W') \ge d$.    
\end{proof}

Let us briefly mention how Lemma~\ref{lemma_construction_4} can be used in order to obtain the best known lower bound for $A_q(9,4;3)$ \cite{cossidente2019subspace, K1} and 
$A_q(10,4;3)$,\cite{K1}. First let $\left(\mathcal{D}_0,\mathcal{D}_1\right)$ be a $(6,4,3)$-sequence. Here $\mathcal{D}_0$ is an LMRD code of cardinality $q^6$, and 
$\mathcal{D}_1$ is a $(6,\star,4;3)_q$-CDC with cardinality $q^6+2q^2+2q$, where we have removed one codeword from a pair of disjoint codewords, see \cite{CP,ubt_eref7922} 
for constructions of CDCs with cardinality $q^6+2q^2+2q+1$. 

As regards $A_q(9, 4; 3)$, we take as code $\mathcal{C}$ a $(6,\star,4;3)_q$-CDC with cardinality $q^6+2q^2+2q+1$, see \cite{CP,ubt_eref7922}. In order to determine 
a distance-partition $\left(\mathcal{C}_0,\mathcal{C}_1\right)$ of $\mathcal{C}$, we need to find a large partial-spread subcode of $\mathcal{C}$. In 
\cite[Theorem 3.12]{cossidente2019subspace}, it is shown that we can choose $\mathcal{C}_0$ of cardinality $q^3-1$ if we choose $\mathcal{C}$ from \cite{CP}. However, 
as shown in \cite{K1}, the same can be done if we choose $\cC$ from \cite{ubt_eref7922}.
This can be made more precise in the language of linearized polynomials.  For \cite[Lemma 12, Example 4]{ubt_eref7922} the representation $\F_q^6\cong \F_{q^3}\times \F_{q^3}$ 
is used and the planes removed from the lifted MRD code correspond to $ux^q-u^qx$ for $u\in\F_{q^3}$, so that the monomials $ax$ for $a\in\F_{q^3}\backslash\{\mathbf{0}\}$ 
correspond to a partial-spread subcode of cardinality $q^3-1$. 
As subcode $\mathcal{A}$ we choose a single $3$-space, so that we obtain
\begin{eqnarray*}
  A_q(9,4;3)&\ge& 1+\#\mathcal{C}_0\cdot\#\mathcal{D}_1+\#\mathcal{C}_1\cdot\#\mathcal{D}_0 \\ 
  &=&1+\left(q^3-1\right)\cdot\left(q^6+2q^2+2q\right)+\left(q^6-q^3+2q^2+2q+2\right)\cdot q^6 \\ 
  &=& q^{12} + 2q^{8} + 2q^{7} + q^6 + 2q^5 + 2q^4 - 2q^2 - 2q + 1.
\end{eqnarray*}  

For $A_q(10,4;3)$ we let $\mathcal{C}$ be the $(7,\star,4;3)_q$-CDC of cardinality $q^8+q^5+q^4+q^2-q$ constructed in \cite[Theorem 4]{honold2016putative}. Again we need to find 
a large partial-spread subcode $\mathcal{C}_0$ of $\mathcal{C}$. Here $\#\mathcal{C}_0=q^4$ can be achieved, see \cite{K1}. 
Thus, we obtain
\begin{eqnarray*}
  A_q(10,4;3)&\ge& 1+\#\mathcal{C}_0\cdot\#\mathcal{D}_1+\#\mathcal{C}_1\cdot\#\mathcal{D}_0 \\
  &=& 1+q^4\cdot\left(q^6+2q^2+2q\right)+\left(q^8+q^5+q^2-q\right)\cdot q^6 \\
  &=& q^{14} + q^{11} + q^{10} + q^8 - q^7 + 2q^6 + 2q^5 + 1.  
\end{eqnarray*}    

The determination of a large partial-spread subcode is mostly the hardest part in the analytic evaluation of the construction of Lemma~\ref{lemma_construction_4}. However, if $\mathcal{C}$ contains an $(n, \star, d; k)$-CDC that is an LMRD as a subcode, then it contains an $(n, m(q,k,n-k,k), 2k; k)$-CDC that is again an LMRD, i.e., a partial-spread subcode.

\begin{theorem}
  \label{thm_12_4_4}
  $A_q(12,4;4)\ge q^{24} + q^{20} + q^{19} + 3q^{18} + 2q^{17} + 3q^{16} + q^{15} + q^{14} + 2q^{12} + q^{11} + 3q^{10} + 2q^{9} + 4q^8 + 2q^7 + 4q^6 + 2q^5 + 3q^4 + q^3 + q^2 + 1$.
\end{theorem}  
\begin{proof}
In order to apply Lemma~\ref{lemma_construction_4}, let $\left(\mathcal{D}_0,\mathcal{D}_1,\mathcal{D}_2\right)$ be an $(8,4,4)$-sequence and let $\mathcal{C}$ be the 
$(8, \star, 4; 4)_q$-CDC with cardinality $q^{12}+q^2(q^2+1)^2(q^2+q+1)+1$ described in Section~\ref{sec_constructions_rank_metric} and yielding \eqref{construction_A_q(8,4;4)}. 
As $\mathcal{C}$ contains an LMRD subcode and a disjoint codeword we can choose a partial-spread subcode $\mathcal{C}_0$ of cardinality $q^4+1$. As distance-partition we use 
$\left(\mathcal{C}_0,\emptyset,\mathcal{C}\backslash\mathcal{C}_0\right)$. For $\mathcal{D}_0$ and $\mathcal{D}_1$ we choose an LMRD code of cardinality $q^{12}$. 
As $\mathcal{D}_2$ we choose the code obtained from $\cC$ by removing one codeword from a pair of disjoint codewords. As regards $\mathcal{A}$, we choose a single $4$-subspace. Thus, we obtain
  \begin{eqnarray*}
    A_q(12,4;4) &\ge& \#\mathcal{A}+\#\mathcal{C}_0\cdot\#\mathcal{D}_2+\#\mathcal{C}_1\cdot\#\mathcal{D}_1+\#\mathcal{C}_2\cdot\#\mathcal{D}_0 \\
    &=& 1+\left(q^{4}+1\right)\cdot\left(q^{12}+q^2(q^2+1)^2(q^2+q+1)\right) \\ 
    &&+\left(q^{12}+q^2(q^2+1)^2(q^2+q+1)-q^4\right)\cdot q^{12} \\ 
    &=& q^{24} + q^{20} + q^{19} + 3q^{18} + 2q^{17} + 3q^{16} + q^{15} + q^{14} + 2q^{12} + q^{11} \\&&+ 3q^{10} + 2q^{9} + 4q^8 + 2q^7 + 4q^6 + 2q^5 + 3q^4 + q^3 + q^2 + 1. 
  \end{eqnarray*}  
\end{proof}
This construction improves upon the recent improvement of \cite[Proposition 4.6]{K} for $A_q(12,4;4)$. The approach of the previous theorem is rather general and universal since many of the largest known CDCs contain an LMRD as a subcode. As an example we consider the $(2k, \star, 4; k)_q$-CDCs from \cite{CP2}.

\begin{theorem}
  \label{thm_2n_n_4}
For a positive integer $k \ge 5$, let $\mathcal{C}$ be the $(2k,\Lambda,4;k)_q$-CDC from \cite[Theorem 3.8]{CP2} or \cite[Theorem 3.11]{CP2}, depending on whether $k$ is even or odd. Then $A_q(3k,4;k)\ge 1+\left(q^k+1\right)\cdot\left(\Lambda-1\right)+\left(\Lambda-q^k-1\right)\cdot q^{k(k-1)}$.   
\end{theorem}   
\begin{proof}
We apply Lemma~\ref{lemma_construction_4} where $\left(\mathcal{D}_0,\mathcal{D}_1,\mathcal{D}_2\right)$ is a $(2k,4,k)$-sequence and $\mathcal{C}$ is the $(2k,\Lambda,4;k)_q$-CDC 
from \cite{CP2}. As $\mathcal{C}$ contains an LMRD subcode and a disjoint codeword we can choose a partial-spread subcode $\mathcal{C}_0$ of cardinality $q^k+1$. As distance-partition 
we use $\left(\mathcal{C}_0,\emptyset,\mathcal{C}\backslash\mathcal{C}_0\right)$. For $\mathcal{D}_0$ and $\mathcal{D}_1$ we choose an LMRD code of cardinality $q^{k(k-1)}$. 
As $\mathcal{D}_2$ we choose the code obtained from $\cC$ by removing one codeword from a pair of disjoint codewords. As $\mathcal{A}$ we choose a single $k$-subspace. Thus, we obtain
  \begin{eqnarray*}
    A_q(3k,4;k) &\ge& \#\mathcal{A}+\#\mathcal{C}_0\cdot\#\mathcal{D}_2+\#\mathcal{C}_1\cdot\#\mathcal{D}_1+\#\mathcal{C}_2\cdot\#\mathcal{D}_0 \\ 
    &=& 1+  \left(q^k+1\right)\cdot\left(\Lambda-1\right) + \left(\Lambda-q^k-1\right)\cdot q^{k(k-1)}.
  \end{eqnarray*}
\end{proof}

The construction described in Lemma~\ref{lemma_construction_4} can be applied recursively, as we are going to see in the next lines for $A_q(16,4;4)$. Let $\mathcal{C}'$ be 
the $(12, \star, 4; 4)$-CDC code yielding the lower bound of $A_q(12, 4; 4)$ of Theorem~\ref{thm_12_4_4}. In order to find a partial-spread subcode of $\mathcal{C}'$, we 
remark that $\mathcal{C}_0$ is a partial-spread subcode of $\mathcal{C}$ of cardinality $q^4+1$. Thus, for each codeword of $\cC_0$, via $\mathcal{D}_2$, we can select $q^4$ 
codewords in $\mathcal{C}'$ that are pairwise disjoint. By adding the elements of $\mathcal{A}$, we end up with a partial-spread subcode $\mathcal{C}'_0$ of $\mathcal{C}'$ 
of cardinality $q^8+q^4+1$. By choosing $\mathcal{A}$ and the $(8,4,4)$-sequence $\left(\mathcal{D}_0,\mathcal{D}_1,\mathcal{D}_2\right)$ as in the proof of Theorem~\ref{thm_12_4_4}, 
and by using the distance-partition $\left(\mathcal{C}'_0,\emptyset,\mathcal{C}'\backslash\mathcal{C}'_0\right)$, Lemma~\ref{lemma_construction_4} gives
\begin{eqnarray*}
  A_q(16,4;4)&\ge& 1+\left(q^8+q^4+1\right)\cdot \#\mathcal{D}_2 + (\#\cC' - q^8 - q^4 - 1)\cdot q^{12}\\ 
  &=& q^{36} + q^{32} + q^{31} + 3q^{30} + 2q^{29} + 3q^{28} + q^{27} + q^{26} + 2q^{24} + q^{23} + 3q^{22} + 2q^{21} \\ &&+ 4q^{20} + 2q^{19} 
  + 4q^{18} + 2q^{17} + 4q^{16} + 2q^{15} + 4q^{14} + 2q^{13} + 5q^{12} + 2q^{11} + 4q^{10}\\&& + 2q^9 + 4q^8 + 2q^7 + 4q^6 + 2q^5 + 3q^4 + q^3 + q^2 + 1, 
\end{eqnarray*} 
where 
$\#\cC'=q^{24} + q^{20} + q^{19} + 3q^{18} + 2q^{17} + 3q^{16} + q^{15} + q^{14} + 2q^{12} + q^{11} + 3q^{10} + 2q^{9} + 4q^8 + 2q^7 + 4q^6 + 2q^5 + 3q^4 + q^3 + q^2 + 1$ 
and $\#\mathcal{D}_2=q^{12}+q^2(q^2+1)^2(q^2+q+1)$.

The next result considers the case of Lemma~\ref{lemma_construction_4} when both $\cD_r$ and $\cA$ contain a partial-spread subcode.

\begin{lemma}
Let $\mathcal{C}$ be a CDC obtained from the construction of Lemma~\ref{lemma_construction_4} with a distance-partition $\left(\mathcal{C}_0,\dots,\mathcal{C}_r\right)$, a $(k+s,d,k)$-sequence $\left(\mathcal{D}_0,\dots,\mathcal{D}_r\right)$, and a CDC $\mathcal{A}$. If $\mathcal{D}_r$ contains a partial-spread subcode $\mathcal{P}$ and $\mathcal{A}$ contains a partial-spread subcode $\mathcal{P}'$, then $\mathcal{C}$ contains a partial-spread subcode of cardinality $\#\mathcal{C}_0\cdot\#\mathcal{P}+ \#\mathcal{P}'$.   
\end{lemma}

Of course we can also apply the construction of Lemma~\ref{lemma_construction_4} on the CDCs constructed in Section~\ref{sec_constructions_rank_metric}. We do this exemplarily for the codes yielding improved lower bounds for $A_q(4k,2k;2k)$ to obtain a lower bound for $A_q(6k,2k;2k)$.

\begin{theorem}
   \begin{eqnarray*}
   A_q(6k,2k;2k) &\ge& 1+  \left(q^{2k}+1\right)\cdot\left(q^{2k(k+1)} + a(q,2k,2k,k,k) + q^{k(k/2+2)} + 2q^k\right) \\ 
   &&+ \left(q^{2k(k+1)} + a(q,2k,2k,k,k) + q^{k(k/2+2)} + 2q^k-q^{2k}\right)\cdot q^{2k(k+1)}
   \end{eqnarray*}
   for even $k\ge 4$.
\end{theorem}
\begin{proof}
Let $k\ge 4$ be a positive even integer. In order to apply Lemma~\ref{lemma_construction_4} let $\left(\mathcal{D}_0,\dots,\mathcal{D}_k\right)$ be a $(4k,4k,4k)$-sequence and let $\mathcal{C}$ be the $(4k,2k;2k)_q$-CDC described in Section~\ref{sec_constructions_rank_metric} and yielding \eqref{construction_A_q(4k, 2k; 2k)}. Hence
$$
\#\mathcal{C} = q^{2k(k+1)} +a(q,2k,2k,k,k)+ q^{k(k/2+2)}+2q^k+1.
$$

As $\mathcal{C}$ contains an LRMD subcode and a disjoint codeword we can choose a partial-spread subcode $\mathcal{C}_0$ of cardinality $q^{2k}+1$. As distance-partition we use $\left(\mathcal{C}_0,\emptyset,\dots,\emptyset,\mathcal{C}\backslash\mathcal{C}_0\right)$. For $\mathcal{D}_0,\dots, \mathcal{D}_{k-1}$ we choose an LMRD code of cardinality $q^{2k(k+1)}$. As $\mathcal{D}_k$ we choose the code $\mathcal{C}$ from above removing one codeword from a 
  pair of disjoint codewords. As $\mathcal{A}$ we choose a single $2k$-subspace. Thus, we obtain
  \begin{eqnarray*}
    A_q(6k,2k;2k) &\ge& \#\mathcal{A}+\#\mathcal{C}_0\cdot\#\mathcal{D}_k+\#\mathcal{C}_k\cdot\#\mathcal{D}_0 \\ 
    &=& 1+  \left(q^{2k}+1\right)\cdot\left(q^{2k(k+1)} +a(q,2k,2k,k,k)+ q^{k(k/2+2)} + 2q^k\right)\\ 
    &&+ \left(q^{2k(k+1)} +a(q,2k,2k,k,k)+ q^{k(k/2+2)} + 2q^k-q^{2k}\right)\cdot q^{2k(k+1)}.
  \end{eqnarray*}  
\end{proof}

\bigskip
{\footnotesize
\noindent\textit{Acknowledgments.}
The work of the first, third and fourth author was supported 
by the Italian National Group for Algebraic and Geometric Structures and their Applications (GNSAGA-- INdAM).}


\end{document}